\documentclass[12pt]{amsart}
\usepackage{amsmath,amsthm}
\numberwithin{equation}{section}
\newtheorem{thm}[equation]{Theorem}
\newtheorem{cor}[equation]{Corollary}
\newtheorem{lm}[equation]{Lemma}
\newtheorem{prp}[equation]{Proposition}

\theoremstyle{definition}
\newtheorem{df}[equation]{Definition}

\newtheorem{prb}[equation]{Problem}
\newtheorem{conj}[equation]{Conjecture}
\theoremstyle{remark}
\newtheorem{rem}[equation]{Remark}
\DeclareMathOperator{\BN}{\mathbb{N}} 
\DeclareMathOperator{\BZ}{\mathbb{Z}} 
\DeclareMathOperator{\BQ}{\mathbb{Q}} 
\DeclareMathOperator{\BR}{\mathbb{R}} 
\DeclareMathOperator{\BC}{\mathbb{C}} 
\DeclareMathOperator{\BH}{\mathbb{H}} 
\DeclareMathOperator{\BE}{\mathbb{E}} 
\DeclareMathOperator{\BG}{\mathbb{G}} 

\DeclareMathOperator{\Norm}{N}        
\DeclareMathOperator{\Twin}{T}
\newcommand{\konj}{\overline}
\newcommand{\norm}[1]{\Vert{#1}\Vert}
\newcommand{\vep}{\varepsilon}
\begin{document}
\title{Sums of squares and orthogonal integral vectors}
\author{Lee M. Goswick, Emil W. Kiss,
G\'abor Moussong, N\'andor Sim\'anyi}
\address[Lee M. Goswick]{The University of Alabama at Birmingham\\
  Department of Mathematics\\
  1300 University Blvd., Suite 452\\
  Birmingham, AL 35294 U.S.A.}
\address[Emil W. Kiss]{E{\"o}tv{\"o}s University\\
  Department of Algebra and Number Theory\\
  1117 Budapest, P{\'a}zm{\'a}ny P{\'e}ter
  s{\'e}t{\'a}ny 1/c\\
  Hungary}
\address[G\'abor Moussong]{E{\"o}tv{\"o}s University\\
  Department of Geometry\\
  1117 Budapest, P{\'a}zm{\'a}ny P{\'e}ter
  s{\'e}t{\'a}ny 1/c\\
  Hungary}
\address[N\'andor Sim\'anyi]{The University of Alabama at Birmingham\\
  Department of Mathematics\\
  1300 University Blvd., Suite 452\\
  Birmingham, AL 35294 U.S.A.}
\email[Lee M. Goswick]{goswick@amadeus.math.uab.edu}
\email[Emil W. Kiss]{ewkiss@math.elte.hu}
\email[G\'abor Moussong]{mg@math.elte.hu}
\email[N\'andor Sim\'anyi]{simanyi@math.uab.edu}
\thanks{Second author supported by Hungarian Nat.\ Sci.\
  Found.\ (OTKA) Grant No.\ NK72523, third author supported
  by Hungarian Nat.\ Sci.\ Found.\ (OTKA) Grant No.\
  T047102, fourth author supported by the National Science
  Foundation, grants DMS-0457168 and DMS-0800538.}



\begin{abstract}
  Two vectors in $\BZ^3$ are called \emph{twins} if they are
  orthogonal and have the same length. The paper describes
  twin pairs using cubic lattices, and counts the number of
  twin pairs with a given length. Integers $M$ with the
  property that each integral vector with length $\sqrt{M}$
  has a twin are called twin-complete. They are completely
  characterized modulo a famous conjecture in number theory.
  The main tool is the decomposition theory of Hurwitz
  integral quaternions. Throughout the paper we made a
  concerted effort to keep the exposition as elementary as
  possible.
\end{abstract}

\maketitle

\parskip=0pt plus 2.5pt

\section{Introduction and main results}
\label{sec_intr}

\noindent
An \emph{icube} in $\BZ^n$ of dimension $k$ is a sequence
$(v_1,\ldots,v_k)$ of $k$ nonzero vectors in~$\BZ^n$ that
are pairwise orthogonal and have the same length. The
subgroup generated by $v_1,\ldots,v_k$ is called the
corresponding \emph{cubic lattice}. The common \emph{length}
of the vectors~$v_i$ is denoted by~$\norm{v_i}$, and is
called the \emph{edge length} of the icube. By the
\emph{norm} of $v_i$ we shall mean~$\Norm(v_i)=\norm{v_i}^2$
(a similar convention is used also for Gaussian integers and
quaternions). A {\em twin pair} is a $2$-dimensional icube
in $\BZ^3$.

In this paper we investigate how icubes can be
\emph{constructed, counted,} and \emph{extended.} We shall
consider the case $n=3$. The main results are the following.
\begin{itemize}
\item Theorem~\ref{count_twins} counts all twin pairs with a
  given edge length.
\item Proposition~\ref{icube_length_int} and
  Corollary~\ref{twin_to_icube} show that a twin pair can be
  extended to a $3$-dimensional icube if and only if its
  edge length is an integer.
\item Theorem~\ref{primitive_unique_big_thm} and
  Corollary~\ref{primitive_unique_big_cor} investigate the
  existence and uniqueness of \hbox{$3$-dimensional} cubic
  lattices containing a single integral vector and the
  extension of single vectors to twins.
\item Theorem~\ref{twin_compl_full} and
  Corollary~\ref{twin_compl_list} characterize 
  twin-complete numbers.
\item The above results are based on the following
  representation theorems:
  \begin{itemize}
  \item[] $k=1$\,: Theorem~\ref{3repr} and
    Theorem~\ref{twin_quad_pow};
  \item[] $k=2$\,: Theorem~\ref{repr_twins};
  \item[] $k=3$\,: Theorem~\ref{e_matrix_prop} and
    Corollary~\ref{sarkozy_cor}.
  \end{itemize}
\end{itemize}
In the rest of the Introduction, we put these results into
context.

\goodbreak

The problem of \emph{construction} and \emph{counting} for
$3$-dimensional icubes in $\BZ^3$ has been solved by
A.~S\'ark\"ozy \cite{Sar61}. To formulate his main result,
we use a construction discovered by Euler.  The following
well-known facts show how to obtain rotations in~$\BR^3$.
Throughout the paper we identify $v=(v_1,v_2,v_3)\in\BR^3$
with the pure quaternion $V(v)=v_1 i+v_2 j+ v_3 k$.

\begin{thm}[see \cite{CS03}, Section~3]\label{euler_matrix}
 Let\/ $\BH^{*} = \BH\setminus\{0\}$ denote the set of 
 nonzero quaternions, and
 $V$ the space of all quaternions with zero real part. 
 For $\alpha \in \BH^{*}$, let $M(\alpha)$ denote the matrix of the
 transformation $\alpha(\,\cdot\,)\alpha^{-1}:\; V \rightarrow V$
 expressed in the standard basis ($i,\,j,\,k$).
 Then there exists a surjective linear representation
 ${\rho: \BH^{*} \rightarrow \text{\sl SO}(3,\BR)}$ such that
\begin{enumerate}
\item[$(1)$] $\ker(\rho) = \BR^{*}$.
\item[$(2)$]The matrix of $\rho(\alpha)$ in the standard basis
 ($i,\,j,\,k$) is
\begin{equation*}
M(\alpha)=
\frac{1}{d}\begin{pmatrix}
m^2+n^2-p^2-q^2  &    -2mq+2np          &   2mp+2nq\\
2mq+2np          &    m^2-n^2+p^2-q^2   &   -2mn+2pq\\
-2mp+2nq         &    2mn+2pq           &   m^2-n^2-p^2+q^2
\end{pmatrix}\,,
\end{equation*}
\end{enumerate}
where $\alpha = m + n i + p j + q k$ and $d = m^2+n^2+p^2+q^2$.
We note that the restriction of the representation $\rho$ to the
unit sphere $S^3$ of\/ $\BH$ is the adjoint representation of\/
$S^3$ with the kernel $\{1,\, -1\}$, being also the universal covering
of the real projective space $\text{\sl SO}(3,\BR)$.
\end{thm}

In what follows, we shall concern ourselves with the \emph{Euler
  matrix} $E(\alpha) = dM(\alpha)$. We are interested in $E(\alpha)$
when its entries are integers. Call such a matrix \emph{primitive} if
the greatest common divisor of its nine entries is~$1$. Similarly, an
icube (or a single integral vector) is primitive if the $nk$ entries
are relatively prime.

\begin{thm}[S\'ark\"ozy, \cite{Sar61}]\label{sarkozy}
  If $m,n,p,q\in\BZ$, then $E(m+ni+pj+qk)$ is primitive if
  and only if $\gcd(m,n,p,q)=1$ and $d$~is odd. Every
  primitive $3$-dimensional icube in~$\BZ^3$ can be obtained
  from such an Euler matrix by permuting columns and
  changing the sign of the third column if necessary.
\end{thm}

This theorem is analyzed in Section~\ref{sec_matrix} and in
Corollary~\ref{primitive_icube_repr}. S\'ark\"ozy went on to
count all $3$-dimensional icubes in~$\BZ^3$ with a given
edge length~$d$.

\medskip

We next look at the question of {\it extension.} Our first
observation puts an obvious limitation on those vectors that
can be extended to a $3$-dimensional icube in~$\BZ^3$.

\begin{prp}\label{icube_length_int}
  Let $(v_1,\ldots,v_n)$ be an $n$-dimensional icube in
  $\BZ^n$. If $n$ is odd, then its edge length is an
  integer.
\end{prp}

\begin{proof}
  Let $d$ denote this length. The volume of the cube is
  $d^n$, which is an integer, since it is the determinant of
  the integer matrix $(v_1,\ldots,v_n)$. We have that $d^2$ 
  is also an integer, since the vectors have integer components,
  which implies $d^{n-1}$ is an integer. Therefore,
  $d=d^n/d^{n-1}$ is rational, and, moreover, an integer.
\end{proof}

This observation makes it easy to answer the following:
which $1$-dimensional icubes (that is, which vectors
in~$\BZ^3$) can be extended to a $3$-dimensional icube? It
turns out that the trivial necessary condition given by
Proposition~\ref{icube_length_int} is sufficient.

\begin{thm}\label{icube_exists}
  A vector in $\BZ^3$ is contained in a $3$-dimensional
  icube if and only if its length is an integer.
\end{thm}

\begin{proof}
  Let $u=(a,b,c)$ be a primitive integral vector, whose
  length $d$ is an integer, so $a^2+b^2+c^2=d^2$. We may
  assume that $a$ is odd. It has been known since at
  least \cite{Car15} that in this case there
  exist $m,n,p,q\in\BZ$ such that $u$~is exactly the first
  column of the corresponding Euler matrix. Thus, the columns
  of this matrix extend $u$ to the desired icube.

  If $x$ is a non-primitive vector of integer length, then
  it can be written uniquely as~$gu$, where $g\in\BZ$ and
  $u\in\BZ^3$ is primitive.  Then the length of $u$ is also
  an integer, so it extends to an icube $(u,v,w)$. Therefore,
  $(gu,gv,gw)$ extends~$x=gu$.
\end{proof}

Theorem~\ref{3repr} also yields Theorem~\ref{icube_exists},
but by using quaternions (see Remark~\ref{3repr_rem}). When
$u$ is primitive, the cubic lattice generated by any
$3$-dimensional icube containing $u$ is always the same (see
Theorem~\ref{primitive_unique_big_thm}).

The next question is this: which $2$-dimensional icubes in
$\BZ^3$ can be extended to a $3$-dimensional icube? Again,
the necessary condition that the length be an integer is
sufficient (see Corollary~\ref{twin_to_icube}).

\medskip

Having surveyed a few results concerning $3$-dimensional
icubes, we now turn our attention to the $2$-dimensional
case. From now on by an icube we shall always mean a
$3$-dimensional icube in $\BZ^3$. The theorems described
below are our results.  The essence of them is that we
understand vectors and twin pairs by putting them into large
$3$-dimensional cubic lattices.

\begin{thm}\label{primitive_unique_big_thm}
  Let $x\in\BZ^3$ have norm $nm^2$, where $n$ is
  square-free. Then there exists an icube $(u,v,w)$ with
  edge length $m$ such that the corresponding cubic
  sublattice contains~$x$. If $x$ is primitive, then this
  cubic lattice is unique, and is given by an Euler matrix
  $E(\alpha)$, for a quaternion $\alpha$ with integer
  coefficients.
\end{thm}

The existence part of this result follows from
Theorem~\ref{3repr} (see Remark~\ref{3repr_rem}). The
uniqueness part is proved at the end of
Section~\ref{sec_twin}, but it is also a consequence of
Corollary~\ref{sarkozy_cor} and Theorem~\ref{3repr}.

If $(u,v,w)$ is an icube and $a,b\in\BZ$, then $(av+bw,
-bv+aw)$ is a twin pair. Theorem~\ref{repr_twins} shows that
\emph{we get all twin pairs this way.} To count all twin
pairs, the corresponding cubic lattice should be made
unique. This is achieved in the same theorem by making the
cubic lattice as large as possible, but not necessarily
as large as in Theorem~\ref{primitive_unique_big_thm} above.
The difficulty is with non-primitive vectors, because there
is no trivial reduction to the primitive case.  For example,
$3(8,-10,9)$ and $7(4,5,2)$ are twins, and this is explained
by the cubic lattice $(u,v,w)=E(2i+j+4k)$, with $a=2$ and
$b=1$.  Theorem~\ref{twin_quad_pow} shows how a vector in
a cubic lattice can be divisible by a prime
``unexpectedly''. Theorem~\ref{repr_twins} describes, using
the language of quaternions, how large this common cubic
lattice really is for a given pair of twins. As an
application, we count all twin pairs with given norm in
Theorem~\ref{count_twins}.

The problem of extending single vectors to twins is more
difficult. A~consequence of our counting result
is that the common norm of twins is always the sum of two
squares.  The converse, however, is not true, as the example
of $(2,2,3)$ shows: its norm is $17=1^2+4^2$, but it does
not have a twin. The case of primitive vectors is
characterized by the following (the proof is at the end of
Section~\ref{sec_twin})

\begin{cor}\label{primitive_unique_big_cor}
  Using the notation of
  Theorem~\ref{primitive_unique_big_thm} suppose that $x$ is
  primitive and $x=au+bv+cw$.
\begin{enumerate}
\item[$(1)$] If none of $a$, $b$, $c$ is zero, then $x$ does
  not have a twin.
\item[$(2)$] If exactly one of $a$, $b$, $c$ is zero, then
  $x$ has exactly two twins. If, say, $a=0$, then these are
  $cv-bw$ and its negative.
\item[$(3)$] If two of $a$, $b$, $c$ are zero, then $x$ has
  exactly four twins, and so is contained in a unique icube.
  If, say, $a=b=0$, then $c=\pm 1$, $n=1$, and the twins of
  $x$ are $\pm u$ and $\pm v$. This case happens exactly
  when the norm of~$x$ is a square.
\end{enumerate}
\end{cor}

\begin{df}\label{twin_compl_df}
  An integer $M>0$ is \emph{twin-complete} if every vector
  in $\BZ^3$ with norm $M$ has a twin, and there is such a
  vector (that is, $M$ is not of the form $4^n(8k+7)$).
\end{df}

\begin{thm}\label{twin_compl_full}
  A positive integer is twin-complete if and only if its 
  square-free part is twin-complete. A positive square-free 
  integer is twin-complete if and only if it can be written 
  as a sum of two squares, but not as a sum of three positive 
  squares.
\end{thm}

We give a complete list of twin-complete numbers modulo the
following conjecture.

\begin{conj}\label{Gauss-conj}
  The complete list of those square-free numbers that can be
  written as a sum of two squares, but not as a sum of three 
  positive squares is the following: 
  $\{1,  2, 5, 10, 13, 37, 58, 85, 130\}$.
\end{conj}

\begin{cor}\label{twin_compl_list}
  The numbers $m^2$, $2m^2$, $5m^2$, $10m^2$, $13m^2$,
  $37m^2$, $58m^2$, $85m^2$, $130m^2$ are twin-complete for
  every integer $m>0$. If Conjecture~\ref{Gauss-conj} holds,
  then there are no other twin-complete numbers.\qed
\end{cor}

The proof of Theorem~\ref{twin_compl_full} is in
Section~\ref{sec_twin_compl}, where many known results
concerning this conjecture are reviewed. The square-free
numbers in question form a subset of Euler's \textit{numeri
  idonei}, and therefore, at most one number can be absent from
the list above. If such an integer does exist, it must
exceed $2\cdot 10^{11}$ \cite{Wei73}, and if it is even, the
Generalized Riemann Hypothesis is false \cite{BC00}.

\begin{prb}\label{main_prb}
  An avenue for future work is to investigate the
  \emph{construction, counting, extension} of higher-dimensional icubes.
\end{prb}

\section{Integral quaternions}
\label{sec_quat}

\noindent
In this section we list some basic results and technical
facts that we shall use in what follows. The general
references about quaternions are \cite{H19}, \cite{HW79},
and \cite{CS03}. The division ring of all quaternions (with
real coefficients) is denoted by~$\BH$. A quaternion is
\emph{pure} if its real part is zero.  Quaternions with
integer coefficients are called \emph{Lipschitz integral
  quaternions}. Such a quaternion is \emph{primitive} if its
coefficients are relatively prime.  Define the special
quaternion $\sigma=(1+i+j+k)/2$.

\begin{prp}[\cite{HW79}]\label{sigma_prop}
  We have $\Norm(\sigma)=1$ and
  $\sigma^2=\sigma-1$. Conjugating by $\sigma$ induces a
  cyclic permutation on $\{i,j,k\}$ (see
  Section~\ref{sec_matrix} for more details).
\end{prp}

Quaternions of the form $a\sigma+bi+cj+dk$ ($a,b,c,d\in\BZ$)
are called \emph{integral quaternions}, or \emph{Hurwitz
  integral quaternions} and they form a ring $\BE$.

\begin{prp}\label{integer_quat_elementary}
  A quaternion $\alpha=a+bi+cj+dk$ belongs to $\BE$ if and only if the
  numbers $2a,2b,2c,2d$ are rational integers with the same parity. If
  $\alpha$ is such, then $\Norm(\alpha)\in\BZ$. A pure integral
  quaternion has integer coefficients, hence the Euler matrix
  $E(\alpha)$, whose columns are $\alpha i\,\konj{\alpha}$, $\alpha
  j\,\konj{\alpha}$, $\alpha k\,\konj{\alpha}$, has integer entries.
\end{prp}

The Hurwitz integral quaternions form a maximal order in the
rational quaternion algebra
$\left(\frac{-1,-1}{\BQ}\right)$. We shall use the
symbol $\mid$ to denote divisibility \emph{on the left} 
in~$\BE$.

\begin{prp}\label{integer_quat_unit}
  An integral quaternion is a unit if and only if its norm
  is~$1$. These are exactly the $24$ elements $\pm 1$, $\pm
  i$, $\pm j$, $\pm k$, $(\pm 1 \pm i \pm j \pm k)/2$, which
  form a group under multiplication. Every integral
  quaternion has a left associate that has integer
  coefficients (\cite{HW79}, p.\ 305, \cite{CS03}, Section
  5.2).
\end{prp}

\begin{thm}\label{quat_eucl}
  The ring $\BE$ is right Euclidean: for every $\alpha,
  \beta\in\BE$ with $\beta\ne 0$, there exist
  $\omega,\rho\in\BE$ such that $\alpha=\beta\omega+\rho$
  and $\Norm(\rho)<\Norm(\beta)$ (see Theorem 373 of
  \cite{HW79}).
\end{thm}

Since $\alpha\mapsto\konj{\alpha}$ is an isomorphism between
$\BE$ and its dual, every assertion that we prove for
$\BE$ remains true if we replace ``left'' with ``right'' and
vice versa. As $\BE$ is left Euclidean, every element can be
written as a product of irreducible quaternions. This
decomposition is unique in a certain sense (see \cite{CS03},
Section 5.2).

We shall need the following two technical
lemmas.

\begin{lm}\label{quat_lnko}
  Suppose that $\alpha\in\BE$ and $p\in\BZ$ is a prime such
  that $p\mid\Norm(\alpha)$ but $p$~does not divide
  $\alpha$. Then $\alpha$ can be written as $\pi\alpha'$,
  where $\Norm(\pi)=p$. This $\pi$ is uniquely
  determined up to right association.

  An element $\pi$ is a left divisor of $\alpha$ with norm
  $p$ if and only if $\pi$ is the generator of the right
  ideal $(\alpha, p)_r$.
\end{lm}

\begin{proof}
  The fact that $\alpha$ is left divisible by a prime $\pi$
  of norm $p$, with the additional property
  $(\pi)_r=(\alpha,p)_r$, follows easily from Theorem~2 in
  Section~5.2 of~\cite{CS03}. (Note that that argument only
  uses the hypothesis that $p$ does not divide $\alpha$, not
  the primitivity of $\alpha$.)  Suppose that $\pi_1$ is
  also left divisor of $\alpha$ with norm $p$. Then $\alpha$
  and $p=\pi_1\konj{\pi_1}$ are in~$(\pi_1)_r$, so
  $(\pi)_r=(\alpha,p)_r\subseteq (\pi_1)_r$.  Thus, $\pi_1$
  divides $\pi$ on the right, and as they have the same
  norm, they are right associates, implying also that
  $(\pi_1)_r=(\pi)_r=(\alpha,p)_r$.
\end{proof}

\begin{lm}\label{twin_primeprop}
  Suppose that $\theta,\eta,\pi\in\BE$, such that
  $\Norm(\pi)=p$ is a prime in~$\BZ$. If $\pi\mid\theta$,
  $p\mid\konj{\theta}\eta$ but $p$ does not divide $\theta$,
  then $\pi\mid\eta$.
\end{lm}

\begin{proof}
  By Lemma~\ref{quat_lnko}, $(p,\theta)_r=(\pi)_r$, that
  is, $\pi = \theta\tau_1+p\tau_2$, for some
  $\tau_1,\tau_2\in\BE$. Hence, $\konj{\pi}\eta =
  \konj{\tau_1}\konj{\theta}\eta+p\konj{\tau_2}\eta$, and as
  $p$ divides $\konj{\theta}\eta$, we get that $p\mid
  \konj{\pi}\eta$. Using $p=\konj{\pi}\pi$, this shows that
  $\pi\mid\eta$.
\end{proof}

\begin{thm}\label{quat_irred}
  An integral quaternion is irreducible in the ring $\BE$ if
  and only if its norm is a prime in $\BZ$ (see Theorem 377
  of \cite{HW79}). The only elements of $\BE$ whose norm is
  $2$ are $\lambda=1+i$ and its left associates. If $p>2$ is
  a prime in $\BZ$, then there exist exactly $24(p+1)$
  integral quaternions whose norm is $p$ (see the note right
  after the proof of Theorem~3 in Section~5.3
  of~\cite{CS03}).
\end{thm}

\begin{cor}\label{quat_int_count}
  The number of integral quaternions with norm~$n$ is $24$
  times the sum of positive odd divisors of~$n$.
\end{cor}

\begin{lm}\label{quat_right_count}
  Let $p\in\BZ$ be a prime and $\ell\ge 0$. Suppose that
  $\pi_1\in\BE$ is fixed and has norm~$p$. Consider all
  integer quaternions $\alpha$ such that
  $\Norm(\alpha)=p^\ell$ and $\alpha\pi_1$ is not
  divisible by~$p$. Then the number of such $\alpha$ is
  $24p^\ell$.
\end{lm}

\begin{proof}
  We do induction on~$\ell$. If $\ell=0$, then the statement
  is trivial, since the number of units is $24$. Suppose
  that $p$ does not divide~$\alpha$.  The dual of
  Lemma~\ref{quat_lnko} shows that $\alpha$ can be
  written as $\alpha_2\pi_2$, with $\Norm(\pi_2)=p$, where
  $\alpha_2$ is unique up to right association, and $\pi_2$
  is unique up to left association.  Apply
  Lemma~\ref{twin_primeprop} for $\theta \mapsto
  \konj{\alpha}$, $\eta\mapsto \pi_1$, and
  $\pi\mapsto\konj{\pi_2}$. We get that if $\alpha\pi_1$ is
  divisible by~$p$, then $\pi_2$ and $\konj{\pi_1}$ are left
  associates.  Conversely, if $\pi_2$ and $\konj{\pi_1}$ are
  left associates, then clearly $p\mid \alpha\pi_1$.

  By Theorem~\ref{quat_irred}, the number of elements of
  norm~$p$ up to left association is $p+1$. So $\pi_2$ can
  be chosen $p$ ways, and by the induction assumption,
  $\alpha_2$ can be chosen $24p^{\ell-1}$ ways for every
  given $\pi_2$. Thus, $\alpha$ can be chosen $24p^{\ell-1}p$
  ways.
\end{proof}

\section{Integral Euler matrices}
\label{sec_matrix}

\noindent
Our goal in this section is to characterize (in
Theorem~\ref{e_matrix_prop}) all Euler matrices $E(\alpha)$
with integer entries (called \emph{integral Euler matrices})
in terms of the corresponding quaternion~$\alpha$. S\'ark\"ozy's
Theorem~\ref{sarkozy} is obtained as Corollary~\ref{sarkozy_cor} 

First, we demonstrate how to permute the columns of an Euler
matrix by changing its generating quaternion. By
Theorem~\ref{euler_matrix}, $E(\alpha)$ is the matrix of
$R(\alpha):\beta\mapsto\alpha\beta\,\konj{\alpha}$, hence
$E(\alpha\vep) = E(\alpha)E(\vep)$. The map $R(\alpha)$ is
always orientation-preserving, but the map corresponding to
an icube (as a matrix) may not be. This problem is
averted by taking the negative of an odd number of columns.

\begin{prp}[cf.\ \cite{CS03}, Section 3.5]\label{matrix_permute}
  Let $\vep$ be
\begin{enumerate}
\item[$(A)$] $\sigma$ or $\sigma^{-1}$, where
  $\sigma=(1+i+j+k)/2$. Then $R(\vep)$ is the rotation of
  $\BR^3$ about the vector $i+j+k$ by an angle of $\pm
  120^{\circ}$ (thus cyclically permuting the three
  coordinate axes). Therefore, $E(\alpha\vep)$ is obtained
  from $E(\alpha)$ by applying a cyclic permutation to the
  columns.
\item[$(B)$] $(1\pm i)/\sqrt{2}$. Then $R(\vep)$ is the
  rotation about the unit vector $i \in \BZ^3$ by an
  angle of $\pm 90^{\circ}$ (interchanging the other
  two coordinate axes). Therefore, $E(\alpha\vep)$ is
  obtained from $E(\alpha)$ by switching the last two
  columns and taking the negative of one. A similar
  statement holds for $(1\pm j)/\sqrt{2}$ and $(1\pm
  k)/\sqrt{2}$.
\item[$(C)$] $\pm i$. Then $R(\vep)$ is the half turn (that
  is, $180^{\circ}$ rotation) about the unit vector $i\in
  \BZ^3$ (fixing all coordinate axes). Therefore,
  $E(\alpha\vep)$ is obtained from $E(\alpha)$ by taking the
  negative of the last two columns. A similar statement
  holds for $\pm j$ and~$\pm k$. This transformation is the
  square of the one described in~$(B)$.
\end{enumerate}
Every non-identical permutation of the columns of
$E(\alpha)$ can be obtained by one of the above
modifications of~$\alpha$, but in case of an odd
permutation one of the columns changes its sign. One can
also change the sign of any two columns.\qed
\end{prp}

Before proceeding, let us review the action of these
isometries on $\BR^3$.

\begin{prp}[cf.\ \cite{CS03}, Section 3.5]\label{unit_group}
  Denote by $H$ the group of units of\/~$\BE$ (see
  Proposition~\ref{integer_quat_unit}), set $Q=\{\pm 1,\pm
  i,\pm j,\pm k\}$ and let $G$ be the subgroup of the
  multiplicative group of $\BH$ generated by $H$ and
  $(1+i)/\sqrt{2}$. Then $G$ contains all the isometries
  investigated in Proposition~\ref{matrix_permute}.  The
  group $G$ has $48$ elements.

  The element $\sigma$ has order~$6$. The rotation
  $R(\sigma)$ maps $\eta=ai+bj+ck$ to $ci+aj+bk$, so it
  permutes the components cyclically.

  The element $(1+i)/\sqrt{2}$ has order~$8$. The
  corresponding rotation $R\big((1+i)/\sqrt{2}\big)$ maps
  $\eta$ to $ai-cj+bk$. The square of this rotation is
  $R(i)$, mapping $\eta$ to $ai-bj-ck$.

  In general, $G$ acts on the set of pure quaternions via
  the rotations $R(\rho)$ with $\rho\in G$. The orbit of
  $\eta$ under $Q$ consists of $\eta$, $-ai-bj+ck$,
  $-ai+bj-ck$, and $+ai-bj-ck$. If we disregard the signs,
  then every other element of~$H$ induces a fixed point free
  permutation on the components of~$\eta$.\qed
\end{prp}

\begin{thm}\label{e_matrix_prop}
  $E(\alpha)$ is a primitive integral Euler matrix if and
  only if the non-zero quaternion $\alpha$ belongs to one of
  the following three types.
  \begin{enumerate}
  \item[$(1)$] $\alpha$ is a primitive Lipschitz integral
    quaternion with an odd norm.
  \item[$(2)$] $\alpha=\beta/\sqrt{2}$, where $\beta$ is a
    primitive Lipschitz integral quaternion such that
    $\Norm(\beta) \equiv 2~(4)$, or equivalently: exactly
    two components of $\beta$ are odd.
  \item[$(3)$] $\alpha=\beta/2$, where $\beta$ is a
    primitive Lipschitz integral quaternion such that
    \hbox{$\Norm(\beta) \equiv 4~(8)$,} or equivalently: all
    four components of $\beta$ are odd (so $\alpha$ is a
    Hurwitz integral quaternion).
  \end{enumerate}
  In all cases $\Norm(\alpha)$ is an odd integer. Each
  column and each row of $E(\alpha)$ contains exactly one
  odd entry. The number of odd entries in the main diagonal
  in types $(1)$, $(2)$, $(3)$ are $3$, $1$, $0$,
  respectively.
\end{thm}

The proof requires four lemmas, whose proofs are elementary
calculations.

\begin{lm}\label{lipschitz_not_div_8}
  If $\beta$ is a primitive Lipschitz integral quaternion,
  then $\Norm(\beta)$ cannot be divisible by $8$. It is
  congruent to $2$ modulo~$4$ if and only if $\beta$ has exactly two
  odd components, and is congruent to $4$ modulo~$8$ if and only if all
  components of $\beta$ are odd.\qed
\end{lm}

\begin{lm}\label{e_matrix odd_norm_lem}
  If $E(\alpha)$ is a primitive integral Euler matrix, then
  $\Norm(\alpha)$ is an odd integer. Each column and each
  row contains exactly one odd entry.\qed
\end{lm}

\begin{lm}\label{alpha_sigma_lem}
  If the quaternion $\alpha = (m+ni+pj+qk)/2$ belongs to
  class $(3)$ of Theorem~\ref{e_matrix_prop}, then either
  $\alpha\sigma$ or $\alpha\sigma^{-1}$ is a quaternion of
  class $(1)$.\qed
\end{lm}

\begin{lm}\label{type_2_multiply_lem}
  Every quaternion $\alpha = (m+ni+pj+qk)/\sqrt{2}$ of class
  $(2)$ can be multiplied on the right by a suitable unit
  $(1+u)/\sqrt{2}$ to transform it to a quaternion of class $(1)$, 
  where $u\in\{i,j,k\}$.\qed
\end{lm}

\begin{proof}[Proof of Theorem~\ref{e_matrix_prop}]
  Put $\alpha=m+ni+pj+kq$ with real numbers $m,n,p,q$, and
  assume that the Euler matrix $E(\alpha)$ given in
  Theorem~\ref{euler_matrix} has integral entries and is
  primitive. By Lemma~\ref{e_matrix odd_norm_lem},
  $\Norm(\alpha) = m^2+n^2+p^2+q^2$ is an integer, and the
  diagonal elements of $E(\alpha)$ are also integers.
  Taking linear combinations of these quadratic forms, we get
  that $4m^2$, $4n^2$, $4p^2$, and $4q^2$ are all integers.
  By adding and subtracting symmetric off-diagonal elements,
  we obtain that $4mn$, $4mp$, $4mq$,
  $4np$, $4nq$, and $4pq$ are integers as well.  Therefore,
  the square-free parts of the non-zero numbers among $4m^2,
  4n^2, 4p^2$, and $4q^2$ are the same. Denote this common
  square-free part by~$r$.  The quaternion
  $\alpha=m+ni+pj+qk$ can be written uniquely in the form
  \begin{equation}\label{alpha_unique}
    \alpha=\frac{k\sqrt{r}}{2}(a+bi+cj+dk),
  \end{equation}
  where $k \in \BN$, $a,b,c,d \in \BZ$, $(a,b,c,d) = 1$.
  Since the matrix $E(\alpha)$ is primitive, neither~$k$
  nor the square-free $r$ can have any odd prime divisor,
  and $k$ (as a power of~$2$) cannot be greater than~$2$.
  Hence, both $k$ and $r$ are elements of the set $\{1,2\}$,
  but $k=r=2$ violates primitivity of $E(\alpha)$.
  Thus, we are left with the cases
  \begin{enumerate}
    \item[$(a)$] $k = 2$, $r = 1$;
    \item[$(b)$] $k = 1$, $r = 2$;
    \item[$(c)$] $k = r = 1$.
  \end{enumerate}
  These correspond exactly to the cases listed as $(1)$,
  $(2)$, and $(3)$ in Theorem~\ref{e_matrix_prop}. Lemma
  \ref{e_matrix odd_norm_lem} shows that $\Norm(\alpha)$ is
  an odd integer. Therefore, Lemma~\ref{lipschitz_not_div_8}
  finishes the proof of one implication of the theorem.

  Assume now that the quaternion $\alpha$ is one of the
  types $(1)$--$(3)$ in Theorem~\ref{e_matrix_prop}. We want to
  show that $E(\alpha)$ is a primitive integer matrix.  By
  Lemmas \ref{alpha_sigma_lem} and \ref{type_2_multiply_lem},
  there exists a suitable quaternion $\vep \in \BH$ with
  $\Norm(\vep) = 1$ such that $\alpha\vep$ is of class $(1)$,
  and by Proposition~\ref{matrix_permute}, we see that
  $E(\alpha\vep) = E(\alpha)E(\vep)$ is a primitive integral
  matrix if and only if $E(\alpha)$ is.

  We show that $E(\alpha\vep)$ is primitive. Since
  $\Norm(\alpha\vep)$ is odd by assumption, the entries in
  the main diagonal of $E(\alpha\vep)$ are odd. Suppose 
  that an odd rational prime $t$ divides all entries of 
  $E(\alpha\vep)$.  Let $\alpha\vep=m+ni+pj+qk$.  
  The simple calculation preceding (\ref{alpha_unique}) 
  shows that $t$ divides the numbers $4m^2, 4n^2, 4p^2$, 
  and $4q^2$, violating the primitivity of $\alpha\vep$. 
  Thus, $E(\alpha\vep) = E(\alpha)E(\vep)$ is indeed a 
  primitive integer matrix.

  We now show the last statement of the theorem.  If
  $\alpha$ belongs to type~$(1)$, then, as we saw above, the
  entries in the main diagonal are odd, while the other
  entries are clearly even. Quaternions of class $(2)$ are
  handled by Lemma~\ref{type_2_multiply_lem}, and these
  correspond to the interchange of two columns by $(B)$ of
  Proposition~\ref{matrix_permute}. The type~$(3)$ case is
  handled by Lemma~\ref{alpha_sigma_lem} and yields a
  cyclic fixed point free permutation of the columns, by
  $(A)$ of Proposition~\ref{matrix_permute}. This completes
  the proof of Theorem~\ref{e_matrix_prop}.
\end{proof}

\begin{cor}\label{sarkozy_cor}
  Consider a primitive icube as the columns of a matrix~$M$.
  Then there exists a Lipschitz integral quaternion $\alpha$
  such that, by permuting the columns of~$E(\alpha)$ and
  changing the sign of the last column if necessary, we
  get~$M$.
\end{cor}

\begin{proof}
  Change the sign of the last column if and only if $M$ is
  orientation-reversing. The new~$M$ can be written as
  $M=E(\alpha)$, for some quaternion~$\alpha$ (with real
  coefficients) by Theorem~\ref{euler_matrix}.  This
  $\alpha$ belongs to one of the three types listed in
  Theorem~\ref{e_matrix_prop}.  Modify $\alpha$ using
  Proposition~\ref{matrix_permute} so that the odd entries
  move to the main diagonal. Then we get a Lipschitz
  integral quaternion by the last statement of
  Theorem~\ref{e_matrix_prop}. If this transformation
  changes the sign of a column other than what was
  initially the third, then use
  Proposition~\ref{matrix_permute} to change the sign of two
  columns to what they were originally.
\end{proof}

\section{A representation of pure integral quaternions}
\label{sec_3sqares}

\noindent
In this section we decompose single pure quaternions.
Geometrically, this means that we find a large cubic lattice
that contains the corresponding vector. Algebraically, a
cubic lattice is the subgroup of all quaternions
$\delta=\alpha\beta\,\konj{\alpha}$, where $\alpha$ is a
fixed Hurwitz integral quaternion and $\beta$ runs over all
pure integral quaternions. The generating icube is given by
$\alpha i\,\konj{\alpha}$, $\alpha j\,\konj{\alpha}$,
$\alpha k\,\konj{\alpha}$.

The primitive case is easier, and is handled by
Theorem~\ref{3repr}. This already implies
Theorem~\ref{icube_exists}, and the existence statement of
Theorem~\ref{primitive_unique_big_thm}.
Theorem~\ref{twin_quad_pow} explains how a vector in a cubic
lattice can be divisible by a prime ``unexpectedly''.  This
will be used in the characterization of twin-complete
numbers, and is also sufficient to obtain a classical result
about counting all vectors of a given length
(Theorem~\ref{s_count}, Corollary~\ref{s_count_cor}).

The results in this section are closely related to those in
\cite{Pal40}, but that paper deals primarily with Lipschitz integral
quaternions.

\begin{lm}\label{3repr_lm}
  Let $\delta\in\BE$ be a pure quaternion and $p\in\BZ$ a
  prime such that $p^2\mid\Norm(\delta)$ but $p$ does not
  divide $\delta$. Then there exists an element $\pi\in\BE$
  whose norm is $p$ such that $\delta=\pi\delta_1\konj{\pi}$,
  for some $\delta_1\in\BE$.
\end{lm}

\begin{proof}
  By Lemma~\ref{quat_lnko}, we get that
  $\delta=\pi\delta_2$, for some $\pi,\delta_2\in\BE$ such
  that $\Norm(\pi)=p$. Then
  $\Norm(\delta)=p\Norm(\delta_2)$. Hence, $p$ divides
  $\Norm(\delta_2)$ but $p$ clearly does not divide
  $\delta_2$. By the dual of Lemma~\ref{quat_lnko}, we
  obtain an element $\pi_1\in\BE$ with norm~$p$ such that
  $\delta_2=\delta_3\pi_1$. Hence, $\delta=\pi\delta_3\pi_1$.
  Taking conjugates, we get
  $\konj{\delta}=\konj{\pi_1}\konj{\delta_3}\konj{\pi}$; 
  however, $\delta$ is a pure quaternion, hence
  $\konj{\delta}=-\delta$. Therefore, $\delta$ is divisible
  by $\pi$ and by $\konj{\pi_1}$ on the left. By the
  uniqueness statement of Lemma~\ref{quat_lnko}, we get
  that $\pi$ and $\konj{\pi_1}$ are right associates.
  Thus, $\delta=\pi\delta_3\pi_1$ can indeed be written as
  $\pi\delta_1\konj{\pi}$.
\end{proof}

\begin{thm}\label{3repr}
  Let $\delta\in\BE$ be a pure quaternion with
  $\Norm(\delta)= nm^2$. Suppose that no integer prime
  divisor of $m$ divides~$\delta$. Then $\delta$ can be
  written as $\alpha\beta\,\konj{\alpha}$, for some
  $\alpha,\beta\in\BE$ such that $\Norm(\alpha)=m$ and
  $\Norm(\beta)=n$. Here $\alpha$ is uniquely determined,
  that is, any two such elements $\alpha$ are right
  associates of each other and the corresponding elements
  $\beta$ are group-conjugates of each other via a unit
  of\/~$\BE$. If $n=1$, then $\beta$ can be chosen freely to
  be any element of $\{\pm i, \pm j, \pm k\}$.
\end{thm}

\begin{proof}
  The existence of $\alpha$ and $\beta$ is easily proven by
  induction on $m$: apply Lemma~\ref{3repr_lm} successively
  for each of the prime divisors of~$m$.

  For the uniqueness assume that
  $\delta=\alpha_1\beta_1\konj{\alpha_1}=
  \alpha_2\beta_2\konj{\alpha_2}$. We use induction on~$m$
  again. If $m=1$, then $\alpha_1$ and $\alpha_2$ are units,
  so they are right associates, and the unit
  $\vep=\alpha_2^{-1}\alpha_1$ satisfies 
  $\vep\beta_1\vep^{-1}=\beta_2$. If $m>1$, then let
  $p\in\BZ$ be a prime divisor of~$m$. Apply
  Lemma~\ref{quat_lnko} to get $\pi_1,\pi_2\in\BE$, with
  $\pi_1\mid\alpha_1$ and $\pi_2\mid\alpha_2$. Then $\pi_1$
  and $\pi_2$ divide $\delta$ on the left, and the
  uniqueness statement of Lemma~\ref{quat_lnko} implies that
  $\pi_1$ and $\pi_2$ are right associates. Thus, if
  $\pi_2=\pi_1\vep$, $\alpha_1=\pi_1\alpha_3$ and
  $\alpha_2=\pi_2\alpha_4$, then
  $\delta'=\alpha_3\beta_1\,\konj{\alpha_3}=
  (\vep\alpha_4)\beta_2\,\konj{\vep\alpha_4}$. By the
  induction hypothesis, $\alpha_3$ and $\vep\alpha_4$ are
  right associates. Similarly, we see that $\alpha_1=\pi_1\alpha_3$ and
  $\alpha_2=\pi_1\vep\alpha_4$ are also right associates.

  If $n=1$, then $\beta$ is a unit in~$\BE$. Since $\beta$
  is a pure quaternion, it is contained in $\{\pm i, \pm j,
  \pm k\}$.  These six elements are group-conjugates of each
  other via a unit by Proposition~\ref{unit_group}. Therefore,
  by taking a right associate of~$\alpha$, we may choose any
  of them to be~$\beta$.
\end{proof}

\begin{rem}\label{3repr_rem}
  Theorem~\ref{3repr} implies Theorem~\ref{icube_exists},
  and the existence statement of
  Theorem~\ref{primitive_unique_big_thm}. (The uniqueness
  part of Theorem~\ref{primitive_unique_big_thm} clearly
  follows from Theorem~\ref{3repr} and
  Corollary~\ref{sarkozy_cor}, but we give a ``pure
  number-theoretic'' proof in Section~\ref{sec_twin}.)
\end{rem}

\begin{proof}
  Let $u$ be a primitive vector and denote by $\delta$ the
  corresponding pure quaternion. Decompose $\delta$ using
  Theorem~\ref{3repr} as $\delta=\alpha\beta\,\konj{\alpha}$
  with $\Norm(\alpha)=m$. Then the cubic lattice
  corresponding to $\alpha$ has edge length $m$ and
  contains~$u$. This yields the existence statement of
  Theorem~\ref{primitive_unique_big_thm}.

  If the length if~$u$ is an integer, then $n=1$ and we may
  assume that $\beta=i$. Then $\alpha j\,\konj{\alpha}$ and
  $\alpha k\,\konj{\alpha}$ extend $u$ to an icube, proving
  Theorem~\ref{icube_exists} in the primitive case. The
  general case obviously follows from this.
\end{proof}

\begin{lm}\label{twin_180rot}
  Let $\beta,\beta_1\in\BE$ be pure quaternions, each of norm
  $n$. If $\beta +\beta_1\ne 0$, then we
  have $(\beta+\beta_1)\beta(\beta+\beta_1)^{-1}=\beta_1$.
\end{lm}

\begin{proof}
  The proof is a straightforward calculation using
  $\beta^2=\beta_1^2=-n$. Instead of presenting it, we
  explain this formula geometrically. Since $\gamma =\beta
  +\beta_1$ is a nonzero pure quaternion, conjugation by
  $\gamma$ acts on $\BR^3$ as half turn about the line
  through $\gamma$, which clearly takes $\beta$ to
  $\beta_1$.
\end{proof}

\begin{lm}\label{twin_quad}
  Let $\pi,\beta\in\BE$ such that $\beta$ is a pure
  quaternion and $p=\Norm(\pi)>2$ is a prime in $\BZ$. Then
  $\pi\beta\pi^{-1}\in\BE$ if and only if there exists an
  integer $h\in\BZ$ such that $\konj{\pi}\mid h+\beta$.
\end{lm}

\begin{proof}
  Suppose that $\konj{\pi}\mid h+\beta$, that is,
  $\konj{\pi}\tau=h+\beta$, for some $\tau\in\BE$. Then
  \[
  p\tau\konj{\pi}=\pi\konj{\pi}\tau\konj{\pi} =\pi
  h\konj{\pi}+\pi\beta\,\konj{\pi}=
  ph+\pi\beta\,\konj{\pi}\,.
  \]
  Hence, $p\mid
  \pi\beta\,\konj{\pi}$, which shows that
  $\pi\beta\pi^{-1}=(\pi\beta\,\konj{\pi})/p$ is indeed an
  integral quaternion.

  To prove the converse, set $\beta_1=\pi\beta\pi^{-1}$ and
  $\tau=\beta+\beta_1$.  We can assume that $\pi$~does not
  divide $\tau$ on the left, as we now show. Let
\[
\beta_2=(i\pi)\beta (i\pi)^{-1}=i\beta_1 i^{-1}\,,
\]
which is still an integral quaternion. It is clearly
sufficient to prove that $\konj{i\pi}\mid h+\beta$, so we
can work with $i\pi$ instead of $\pi$ in the argument below.
If, however, both $\pi$ and $i\pi$ are ``bad'', that is,
$\pi\mid \tau=\beta+\beta_1$ and 
$i\pi\mid\beta+\beta_2$, then $\pi\mid i^{-1}\beta
i+\beta_1$, which implies $\pi\mid \beta-i^{-1}\beta i$. Put
$\beta=ai+bj+ck$. Then $\beta-i^{-1}\beta i=2(bj+ck)$.  If
$j\pi$ and $k\pi$ are also ``bad'', then $\pi$ divides
$2(ai+ck)$ and $2(ai+bj)$ as well. Taking norms, we get,
using $\Norm(\pi)=p>2$, that $p$ divides $a^2+b^2$,
$a^2+c^2$, and $b^2+c^2$. Thus, $p$ divides $a$,~$b$,~$c$, and, 
finally, $p$~divides~$\beta$. Therefore, $\konj{\pi}\mid
h+\beta$, for $h=0$, and we are done in this case. We
can then indeed assume that $\pi$ does not divide~$\tau$ 
on the left.

Lemma~\ref{twin_180rot} implies that
$\tau\beta\tau^{-1}=\beta_1=\pi\beta\pi^{-1}$, so
$\tau^{-1}\pi$ centralizes $\beta$. Let $d=\Norm(\tau)$.
Then $d\tau^{-1}\pi=\konj{\tau}\pi$ centralizes $\beta$ as
well. The centralizer of $\beta$ consists of elements
$r+s\beta$, where $r,s\in\BR$. This set is closed under
conjugation, since $\beta$ is a pure quaternion, and
therefore, contains $\konj{\konj{\tau}\pi}=\konj{\pi}\tau$.
If we write $\konj{\pi}\tau=u+v\beta$, where $u$ and $v$ are real,
and $\beta=d\beta'$, where $d\in\BZ$ and $\beta'$ is
primitive, then $2u$ and $2vd$ are integers. (We need the
factor $2$, because the integral quaternion~$\konj{\pi}\tau$
need not have integer coefficients).

We now show that $p$ does not divide~$2vd$. Suppose it
does.  Taking norms, we get that $\Norm(\konj{\pi})=p\mid
4\Norm(u+v\beta)=(2u)^2+(2vd)^2\Norm(\beta')$, so $p\mid
2u$. Thus, either $u+v\beta$ has integer coefficients, which
are divisible by $p$, or $2u+2v\beta$ has odd integer
coefficients that are divisible by~$p$. Since $p\ne 2$, we
have $u'+v'\beta=(u+v\beta)/p\in\BE$ in either case. Then
$\konj{\pi}\tau=u+v\beta=
p(u'+v'\beta)=\konj{\pi}\pi(u'+v'\beta)$, and
$\tau=\pi(u'+v'\beta)$, contradicting our assumption that
$\pi$~does not divide $\tau$ on the left. Therefore, we 
have that $p$ does not divide~$2vd$.

Let $x,y$ be integers such that $(2vd)x+yp=1$. Then
\[
\konj{\pi}\mid 2x(u+vd\beta')=x(2u)+(1-yp)\beta'\,.
\]
Since $\konj{\pi}\mid p$, we get that $\konj{\pi}\mid
x(2u)+\beta'$ and take $h=x(2u)d$.
\end{proof}

\begin{thm}\label{twin_quad_pow}
  Suppose that $\alpha,\beta\in\BE$ such that $\beta$ is a
  pure quaternion and $p\in\BZ$ is a prime. Then $p\mid
  \alpha\beta\,\konj{\alpha}$ if and only if one of the
  following cases holds.
\begin{enumerate}
\item[$(1)$] $p$ divides $\alpha$ or $\beta$.
\item[$(2)$] $p=2$ and does not divide $\alpha,\beta$, but
  divides $\Norm(\alpha)$.
\item[$(3)$] $p>2$ and does not divide $\alpha,\beta$, but
  divides $\Norm(\alpha)$, and there exists a right divisor
  $\pi$ of $\alpha$ with norm~$p$ and an integer $h\in\BZ$
  such that $\konj{\pi}\mid h+\beta$.
\end{enumerate}
In particular, every prime divisor
of~$\alpha\beta\,\konj{\alpha}$ divides either $\beta$
or~$\Norm(\alpha)$.
\end{thm}

\begin{proof}
  If $(1)$ holds, then clearly $p\mid
  \alpha\beta\,\konj{\alpha}$.  If $(2)$ holds, then the
  dual of Lemma~\ref{quat_lnko} shows that $\alpha$ is right
  divisible by~$1+i$ (since this is the only element in
  $\BE$ of norm~$2$ up to left association), and
  Proposition~\ref{unit_group} yields that
  $(1+i)\beta\,\konj{1+i}$ is divisible by~$2$. Finally, if
  $(3)$ holds, then $p\mid \pi\beta\,\konj{\pi}$ by
  Lemma~\ref{twin_quad}. This proves one direction of the
  theorem.

  Now assume that $p\mid \alpha\beta\,\konj{\alpha}$ but
  $\alpha$ and $\beta$ are not divisible by~$p$. If $p$ does
  not divide $\Norm(\alpha)$, then we have
  $p\mid\konj{\alpha}(\alpha\beta\,\konj{\alpha})\alpha=
  \Norm(\alpha)^2\beta$. Hence, $p\mid\beta$, which is a
  contradiction. Therefore, we may assume that we are in
  case~$(3)$, that is, $p>2$ and $p\mid\Norm(\alpha)$. We
  proceed by induction on $\Norm(\alpha)$. By the dual of
  Lemma~\ref{quat_lnko}, $\alpha=\alpha_1\pi$ for some
  $\alpha_1,\pi\in\BE$, with $\Norm(\pi)=p$. We show that
  $\beta_1=\pi\beta\,\konj{\pi}$ is divisible by~$p$. Then
  we are clearly done by Lemma~\ref{twin_quad}.

  Suppose $\beta_1$ is not divisible by~$p$. 
  Apply the induction hypothesis to 
  $\alpha_1\beta_1\,\konj{\alpha_1}$ (which is equal to
  $\alpha\beta\,\konj{\alpha}$). We must be in case~$(3)$,
  since $p>2$ and $p$ does not divide both $\alpha_1$ and
  $\beta_1$. Therefore, there exists a right divisor $\pi_1$
  of $\alpha_1$ of norm~$p$ and an integer $h_1$ such that
  $\konj{\pi_1}\mid h_1+\beta_1$.  Taking norms, we see that
  $p=\Norm(\konj{\pi_1})\mid
  \Norm(h_1+\beta)=h_1^2+\Norm(\beta_1)$; however,
  $\Norm(\beta_1)=p^2\Norm(\beta)$ is divisible by~$p$, so
  $p\mid h_1$, and therefore, $\konj{\pi_1}\mid\beta_1$.
  Since $\beta_1$ is not divisible by~$p$, the uniqueness
  part of Lemma~\ref{quat_lnko} shows that $\konj{\pi_1}$
  and $\pi$ are right associates. This implies $\alpha$ is
  divisible on the right by $\pi_1\konj{\pi_1}=p$,
  contradicting our assumptions.
\end{proof}

\begin{prp}\label{twin_no_primitives}
  Suppose that $\beta\in\BE$ is a pure quaternion and $p$ is
  a prime not dividing~$\beta$. Denote by $e$ the number of
  different quaternions of the form $\vep\beta\vep^{-1}$,
  where $\vep$ runs over the units of~$\BE$. Consider all
  quaternions~$\alpha$ whose norm is $p^\ell$, with some
  fixed $\ell>0$. If $p>2$, then the number of quaternions
  of the form $\alpha\beta\,\konj{\alpha}$ that are not
  divisible by~$p$ is
\begin{enumerate}
\item[$(1)$] $ep^\ell$, if $p\mid \Norm(\beta)$;
\item[$(2)$] $e(p^\ell-p^{\ell-1})$, if $-\Norm(\beta)$ is a
  quadratic residue mod~$p$;
\item[$(3)$] $e(p^\ell+p^{\ell-1})$ otherwise.
\end{enumerate}
If $p=2$, then this number is~$0$.
\end{prp}

\begin{proof}
  If $p=2$ (and $\ell>0$), then Theorem~\ref{twin_quad_pow}
  shows that $\alpha\beta\,\konj{\alpha}$ is divisible
  by~$p$, so suppose that $p$ is odd. We call a pair
  $(\alpha_1,\beta_1)$ ``good'', if $\alpha_1\in\BE$ with
  $\Norm(\alpha_1)=p^\ell$ and there is a unit $\vep\in\BE$
  such that $\beta_1=\vep^{-1}\beta\vep$, and $p$ does not
  divide $\alpha_1\beta_1\,\konj{\alpha_1}$. By the
  uniqueness part of Theorem~\ref{3repr}, every element
  $\alpha_1\beta_1\,\konj{\alpha_1}$ is given by exactly
  $24$ pairs. Therefore, it is sufficient to count the good
  pairs for any given~$\beta_1$.

  Let $(\alpha_1,\beta_1)$ be a good pair. Then clearly
  $\alpha_1$ is not divisible by~$p$, so we can write
  $\alpha_1=\alpha_2\pi_2$ by Lemma~\ref{quat_lnko}, where
  $\pi_2$ of norm~$p$ is uniquely determined up to left
  association. Theorem~\ref{twin_quad_pow} shows that
  $\alpha_1\beta_1\konj{\alpha_1}$ is divisible by~$p$ if
  and only if $\konj{\pi_2}$ divides $h+\beta_1$, for some
  integer~$h$ (assuming that $\alpha_1$ is not divisible
  by~$p$). We now count the number of such
  quaternions~$\pi_2$.

  Clearly, $\konj{\pi_2}\mid h+\beta_1$ implies
  $\Norm(\pi_2)=p\mid \Norm(h+\beta_1)=h^2+\Norm(\beta_1)$.
  This means that either $p\mid \Norm(\beta_1)$ or
  $-\Norm(\beta_1)$ is a quadratic residue mod~$p$. Hence, in
  case~$(3)$ above there is no such~$\pi_2$.  Since $p$ does
  not divide~$\beta_1$, it does not divide $h+\beta_1$.
  Thus, by the uniqueness statement of Lemma~\ref{quat_lnko},
  there is exactly one left divisor $\konj{\pi_2}$ up to
  right association with norm~$p$ of any given $h+\beta_1$
  for which $p \mid h^2+N(\beta_1)$, and $\pi_2$ is unique
  up to left association. Clearly, the numbers $h_1$ and
  $h_2$ yield the same~$\konj{\pi_2}$ if and only if
  $h_1\equiv h_2~(p)$. If $p\mid \Norm(\beta_1)$, then $h=0$
  is the only possibility. This yields one ``bad'' value for
  $\pi_2$ up to left association. Otherwise, there are
  exactly two values $1\le h\le p-1$ such that $p\mid
  h^2+\Norm(\beta_1)$ (since $p$~is an odd prime, assuming,
  of course, that $-\Norm(\beta_1)$ is a quadratic residue
  mod~$p$). So, in this case there are two ``bad'' values for
  $\pi_2$ up to left association.

  By Theorem~\ref{quat_irred}, the number of possible choices
  for $\pi_2$ is $p+1$ up to left association. Thus, for
  every given $\beta_1$, the number of good values for
  $\pi_2$ is $p$, $p-1$ and $p+1$, respectively,
  corresponding to cases $(1)$, $(2)$ and $(3)$ in the
  claim. If $\pi_2$ is fixed, then the number of choices for
  $\alpha_2$ so that $\alpha_1$ is not divisible by~$p$ is,
  by Lemma~\ref{quat_right_count}, $24p^{\ell-1}$. Since the
  number of possible $\beta_1$ is~$e$, we get the result.
\end{proof}

In the theorem below, $(-n/p)$ denotes the Legendre symbol
(which is defined to be~$0$ if $p\mid n$).

\begin{thm}\label{s_count}
  Suppose that $m,n\ge 1$ are integers and $n$ is
  square-free. If $m$ is odd, then the number $p(nm^2)$ of
  primitive vectors $(x,y,z)$ whose norm is $nm^2$ is
  \[
  p(nm^2)=p(n)\prod \Big(p^\ell-(-n/p)p^{\ell-1}\Big)\,,
  \]
  where $p^\ell$ runs over the prime powers in the canonical
  form of~$m$. If $m$ is even, then $p(nm^2)=0$.
\end{thm}

\begin{proof}
  We proceed by induction on the number of prime divisors of~$m$.
  Suppose that $m=p^\ell m_1$, where $p$ does not divide~$m_1$.
  Theorem~\ref{3repr} implies that the pure quaternion
  $\delta = xi+yj+zk$ corresponding to $(x,y,z)$ can be
  represented as $\alpha\beta\,\konj{\alpha}$, where
  $\Norm(\alpha)=p^\ell$.

  Theorem~\ref{twin_quad_pow} shows that if $\beta$ is
  primitive, then the only possible prime divisor of
  $\alpha\beta\,\konj{\alpha}$ is~$p$, and if $p=2$ and
  $\ell>0$, then $\alpha\beta\,\konj{\alpha}$ is not
  primitive, because it is divisible by~$2$. Thus, if $p>2$,
  then $\alpha\beta\,\konj{\alpha}$ is primitive if and only
  if $\beta$ is primitive and $\alpha\beta\,\konj{\alpha}$
  is not divisible by~$p$. The formula for $p(nm^2)$ then
  follows clearly from Proposition~\ref{twin_no_primitives}.
\end{proof}

The following is a well-known formula (see (29) in
\cite{Pal40}), and follows with some effort from
Theorem~\ref{s_count}.

\begin{cor}\label{s_count_cor}
  Suppose that $m,n\ge 1$ are integers and $n$ is
  square-free. The number $s(nm^2)$ of all vectors of norm
  $nm^2$ is
  \[
  s(nm^2)=s(n)\prod
  \Big(\sigma(p^\ell)-(-n/p)\sigma(p^{\ell-1})\Big)\,,
  \]
  where $p^\ell$ runs over the odd prime powers in the
  canonical form of~$m$ and $\sigma(s)$ denotes the sum of
  positive divisors of any integer~$s$.
\end{cor}

\section{A parameterization of twin pairs}
\label{sec_twin}

\noindent
Theorem \ref{repr_twins} is our main characterization of
twins. In Theorem \ref{count_twins}, we count twin pairs with
a given norm. Finally, we deal with the problem of extension.
In Corollary~\ref{twin_to_icube}, we show that each pair of
twins whose length is an integer extends to an icube. Then,
at the end of the section, we prove
Theorem~\ref{primitive_unique_big_thm} and
Corollary~\ref{primitive_unique_big_cor}.

Recall that every vector $v=(v_1,v_2,v_3)\in\BR^3$ is
identified with the pure quaternion $V(v)=v_1 i+v_2 j+ v_3
k$.

\begin{prp}\label{twin_quat}
  Two vectors $v$ and $w$ in $\BZ^3$ are twins if and only
  if $\theta=V(v)$ and $\eta=V(w)$ satisfy the following
  conditions.
\begin{enumerate}
\item[$(1)$] $\theta$ and $\eta$ are nonzero pure
  quaternions;
\item[$(2)$] $\Norm(\theta)=\Norm(\eta)$;
\item[$(3)$] $\theta\eta$ is also a pure quaternion;
\item[$(4)$] $\theta$ and $\eta$ have integer
  coefficients.
\end{enumerate}
We call a pair of such quaternions $(\theta,\eta)$ a
\emph{twin pair}.
\end{prp}

\begin{proof}
  It is easy to verify that the real part of
  $V(v)V(w)$ equals the negative of the
  dot product of $v$ and $w$, and the pure
  quaternion part of $V(v)V(w)$ corresponds to
  the cross product of $v$ and $w$.
\end{proof}

We now translate the construction of twin pairs given in
the Introduction. Denote by $(u,v,w)$ the columns of an
Euler matrix given by $\alpha\in\BE$. By
Proposition~\ref{integer_quat_elementary}, $\Norm(\alpha)$
and the components of $(u,v,w)$ are integers.  Let
$z=a+bi\in\BG$ (the ring of Gaussian integers). Then
$E(\alpha)$ maps $i$ to $V(u)$, $j$ to $V(v)$, and $k$ to
$V(w)$. It maps the twin quaternions $zj=aj+bk$ and
$zk=ak-bj$ to $\alpha zj\,\konj{\alpha}$ and $\alpha
zk\,\konj{\alpha}$, respectively, which correspond
to the twin pair $(av+bw, -bv+aw)$.

\begin{df}\label{eq_for_pairs}
  We say that $(\theta,\eta)$ is \emph{parameterized} by the
  pair $(\alpha,z)\in\BH\times\BC$ if $\theta=\alpha z
  j\,\konj{\alpha}$ and $\eta=\alpha z k\,\konj{\alpha}$.
  Two pairs in $(\alpha_1,z_1)\in\BH\times\BC$ and
  $(\alpha_2,z_2)\in\BH\times\BC$ are \emph{equivalent} if
  they parameterize the same $(\theta,\eta)$, that is, if
  $\alpha_1 z_1 j\,\konj{\alpha_1}=\alpha_2 z_2
  j\,\konj{\alpha_2}$ and $\alpha_1 z_1
  k\,\konj{\alpha_1}=\alpha_2 z_2 k\,\konj{\alpha_2}$.
\end{df}

\begin{prp}\label{parametr=>twins}
  Suppose that $(\theta,\eta)$ is parameterized by a pair
  $(\alpha,z)\in\BH\times\BC$. Then
  $\theta\eta=\Norm(\alpha)\Norm(z)\alpha i\,\konj{\alpha}$,
  so $\theta$ and $\eta$ satisfy $(1)$--$(3)$ of
  Proposition~\ref{twin_quat}. If $\alpha\in\BE$
  and $z\in\BG$, then $\theta$ and $\eta$ have integer
  coefficients and are~twins.
\end{prp}

\begin{proof}
  Note that $\alpha^{-1}=\konj{\alpha}/\Norm(\alpha)$ and
  $x\mapsto\alpha x\alpha^{-1}$ is an automorphism of the
  division ring~$\BH$. Therefore,
  $\theta\eta=\Norm(\alpha)\alpha zjzk\,\konj{\alpha}$; however,
  $zjzk=zjzj^{-1}jk=z\konj{z}i$, so the quaternion
  $\theta\eta=\Norm(\alpha)\Norm(z)\alpha i\,\konj{\alpha}$
  is pure.
\end{proof}

\begin{thm}\label{repr_twins}
The characterization of twin quaternions is given by the following.
\begin{enumerate}
\item[$(1)$] The quaternions $\theta$ and $\eta$ are twins
  if and only if $(\theta,\eta)$ is parameterized by a pair
  in $\BE\times\BG$ (whose components are nonzero).
\item[$(2)$] Every pair in $\BE\times\BG$ is equivalent to a
  pair, where the second component is \emph{square-free} in
  $\BG$.
\item[$(3)$] Let $(\alpha_1,z_1), (\alpha_2,z_2)\in
  \BE\times\BG$ be such that both $z_1$ and~$z_2$ are
  square-free. Then these pairs are equivalent if and only
  if there exists a unit $\rho\in\BG$ (that is, an element
  of $\{\pm 1,\pm i\}$) such that $\alpha_2=\alpha_1\rho$
  and $z_2=\rho^2z_1$.
\item[$(4)$] The length of the twins $\theta$ and $\eta$ is
  an integer if and only if in the parameterization $(\alpha,
  z)$ of $(\theta,\eta)$, where $z$ is square-free, $z\in\BG$ 
  is either real or pure imaginary.
\end{enumerate}
\end{thm}

The condition that $z$ is square-free expresses the fact that
the cubic lattice given by $\alpha$ is as large as possible.
We prove this theorem through a series of assertions.

\begin{lm}\label{twin_div_p}
  Suppose that $\theta,\eta\in\BE$ is such that $\theta$,
  $\eta$ and $\theta\eta$ are pure quaternions. Let
  $p\in\BZ$ be a prime such that $p$ divides $\Norm(\theta)$
  and $\Norm(\eta)$. Then $p\mid \theta\eta$.
\end{lm}

\begin{proof}
  Suppose that $p$ does not divide $\theta\eta$. By
  Lemma~\ref{quat_lnko} and Lemma~\ref{3repr_lm}, we have
  $\theta=\pi_1\theta_1$, $\eta=\eta_1\pi_2$, and
  $\theta\eta=\pi\delta_1\konj{\pi}$, where
  $\pi,\pi_1,\pi_2$ have norm~$p$. By the uniqueness part of
  Lemma~\ref{quat_lnko}, $\pi$, $\pi_1$, and $\konj{\pi_2}$
  are right associates; however, $\theta$, $\eta$, and
  $\theta\eta$ are pure quaternions, so
  $\theta\eta=-\konj{\theta\eta}=
  -\konj{\eta}\,\konj{\theta}=-\eta\theta=
  -\eta_1\pi_2\pi_1\theta_1$ is divisible by~$p$, a
  contradiction.
\end{proof}

\begin{lm}\label{twin_common_conj}
  Suppose $\theta,\eta\in\BE$ is such that $\theta$,
  $\eta$, and $\theta\eta$ are pure quaternions. Let
  $p\in\BZ$ be a prime such that $p^2$ divides both
  $\Norm(\theta)$ and $\Norm(\eta)$, but $p$ does not divide
  both $\theta$ and $\eta$. Then there exist elements
  $\pi,\theta_1,\eta_1\in\BE$ such that $\Norm(\pi)=p$,
  $\theta=\pi\theta_1\konj{\pi}$, and
  $\eta=\pi\eta_1\konj{\pi}$.
\end{lm}

\begin{proof}
  By Lemma~\ref{twin_div_p}, $p\mid \delta=\theta\eta$.
  Using Lemma~\ref{3repr_lm}, we can write
  $\theta=\pi\theta_1\konj{\pi}$ and
  $\eta=\pi_1\eta_1\konj{\pi_1}$, where
  $\Norm(\pi)=\Norm(\pi_1)=p$. Since $\theta$ is pure,
  $\delta=\theta\eta=-\konj{\theta}\eta$, and
  Lemma~\ref{twin_primeprop} shows that $\pi\mid\eta$.
  Applying the uniqueness part of Lemma~\ref{quat_lnko}
  to~$\eta$, we obtain that $\pi$ and $\pi_1$ are right
  associates.
\end{proof}

\begin{lm}\label{twin_common_conj1}
  Suppose $\theta,\eta\in\BE$ is such that $\theta$,
  $\eta$, and $\theta\eta$ are pure quaternions. Let
  $p\in\BZ$ be a prime such that $p^2$ divides
  $\Norm(\theta)$ but $p$ does not divide $\theta$. Then
  there exist elements $\pi,\theta_1,\eta_1\in\BE$ such that
  $\Norm(\pi)=p$, $\theta=\pi\theta_1\konj{\pi}$, and
  $p\eta=\pi\eta_1\konj{\pi}$.
\end{lm}

\begin{proof}
  Again, write $\theta=\pi\theta_1\konj{\pi}$. Suppose first
  that $\pi\mid\eta$, that is, $\eta=\pi\eta_2$, for some
  $\eta_2\in\BE$. Then $p\eta=\pi(\eta_2\pi)\konj{\pi}$, and
  we are done in this case. So, we can assume that $\pi$ does
  not divide $\eta$. Since $\theta$ is pure,
  $\delta=\theta\eta=-\konj{\theta}\eta$, and
  Lemma~\ref{twin_primeprop} shows that $p$ does not divide
  $\delta$. By Lemma~\ref{3repr_lm} and the uniqueness part
  of Lemma~\ref{quat_lnko}, we have
  $\delta=\pi\delta_1\konj{\pi}$. Then
  $\pi\theta_1\konj{\pi}\eta=\delta=\pi\delta_1\konj{\pi}$
  implies that $\theta_1\konj{\pi}\eta\pi=
  \delta_1\konj{\pi}\pi=p\delta_1$. Hence, $p\mid
  \konj{\konj{\pi}\eta\pi}\,\theta_1$.
  Lemma~\ref{twin_primeprop} shows that either
  $\konj{\pi}\mid\theta_1$ or $p\mid \konj{\pi}\eta\pi$. The
  first case is impossible because then $\theta$ would be
  divisible by~$p$. If we let $\konj{\pi}\eta\pi=p\eta_1$, then
  $p\eta=\pi\eta_1\konj{\pi}$.
\end{proof}

\begin{prp}\label{repr_exists}
  Every pair $(\theta,\eta)$ of twins is parameterized by a
  pair $(\alpha,z)\in\BE\times\BG$.
\end{prp}

\begin{proof}
  Let $M=\Norm(\theta)=\Norm(\eta)$. Consider all
  representations of the form
  $\theta=\alpha\theta_1\konj{\alpha}$ and
  $\eta=\alpha\eta_1\konj{\alpha}$, where
  $\alpha,\theta_1,\eta_1\in\BE$. Clearly, $\theta_1$ and
  $\eta_1$ are twins as well. There exists such a
  representation (with $\alpha=1$), and so there is one with
  $\Norm(\alpha)$ as large as possible. We present
  reductions to increase $\Norm(\alpha)$.

  If $p$ is a prime such that $p^2$ divides both
  $\Norm(\theta_1)$ and $\Norm(\eta_1)$, but $p$ divides
  neither~$\theta_1$ nor~$\eta_1$, then
  Lemma~\ref{twin_common_conj} allows us to replace $\alpha$
  with $\alpha\pi$. If this $p$ does not divide~$\theta_1$
  but divides $\eta_1$, then we write $\eta_1=p^\ell\eta'$ such
  that $\eta'$ is not divisible by $p$, and apply
  Lemma~\ref{twin_common_conj1} to $\theta_1$ and $\eta'$.
  We again obtain a suitable $\pi$ by using up a
  factor of~$p$ out of $p^\ell$.

  If none of these reductions can be performed further, then
  $\theta_1=d\theta_2$ and $\eta_1=d\eta_2$, for some
  $d\in\BZ$ such that $\Norm(\theta_2)=\Norm(\eta_2)$ is
  square-free. By Lemma~\ref{twin_div_p}, every integer prime
  divisor of $\Norm(\eta_2)$ divides $\theta_2\eta_2$.
  Hence, the square-free integer $\Norm(\eta_2)$ divides
  $\theta_2\eta_2$, whose norm is $\Norm(\eta_2)^2$.
  Therefore, $\theta_2\eta_2=\Norm(\eta_2)\vep$, where $\vep$
  is a unit of~$\BE$. We may assume $\vep=i$, by the
  argument in the last paragraph of the proof of
  Theorem~\ref{3repr}. Thus, $\eta_2=\theta_2i$, and as
  $\theta_2$ and $\eta_2$ are pure quaternions,
  $\theta_2=z_1j$, for some $z_1\in\BG$.  Then $\theta=\alpha
  (dz_1)j\,\konj{\alpha}$ and $\eta=\alpha
  (dz_1)k\,\konj{\alpha}$, proving the claim.
\end{proof}

\begin{lm}\label{square_root}
  Let $(\alpha,z)\in\BH\times\BC$ and suppose that $z=s^2t$,
  for some $s,t\in\BC$. Then the pairs $(\alpha,z)$ and
  $(\alpha s, t)$ are equivalent.
\end{lm}

\begin{proof}
  Since $j s j^{-1} = \konj{s}$ for every $s\in\BC$, we have
  that $zj = st j\,\konj{s}$ and similarly, $zk = st
  k\,\konj{s}$. Therefore, $\alpha z j\,\konj{\alpha}=(\alpha
  s)tj(\konj{\alpha s})$ and $\alpha z
  k\,\konj{\alpha}=(\alpha s)tk(\konj{\alpha s})$.
\end{proof}

This lemma immediately implies $(2)$ of
Theorem~\ref{repr_twins} ($(1)$ has been proven in
Proposition~\ref{repr_exists}). We now proceed to
prove~$(3)$. If a suitable $\rho$ in $(3)$ exists, then
$(\alpha_1,z_1)$ is equivalent to
$(\alpha_1\rho,z_2)=(\alpha_2, z_2)$, by
Lemma~\ref{square_root}. We now only have to prove that if
$(\alpha_1,z_1)$ and $(\alpha_2,z_2)$ are equivalent pairs
in $\BE\times\BG$, then a suitable $\rho$ exists.

Choose elements $s_r\in\BC$ satisfying $s_r^2=z_r$.
Lemma~\ref{square_root} shows that $(\alpha_r,z_r)$ is
equivalent to $(\alpha_r s_r,1)$. From $(\alpha_1 s_1)
j(\konj{\alpha_1 s_1})=(\alpha_2 s_2)j(\konj{\alpha_2 s_2})$,
we get that $(\alpha_2 s_2)^{-1}(\alpha_1s_1)$ centralizes
$j$ (as well as $k$, by the same calculation).
Since the elements of $\BH$ centralizing both $j$ and $k$
are exactly the real numbers, we get that $t=(\alpha_2
s_2)^{-1}(\alpha_1s_1)$ is contained in $\BR$.  This can be
written as $t\alpha_1^{-1}\alpha_2=s_1s_2^{-1}$.  From
$(\alpha_1 s_1) j(\konj{\alpha_1 s_1})=(\alpha_2 s_2)
j(\konj{\alpha_2 s_2})$, we get that
$\Norm(\alpha_1)^2\Norm(s_1)^2=
\Norm(\alpha_2)^2\Norm(s_2)^2$. Hence $\Norm(t)=1$, and 
$t=\pm 1$.

This implies that $\alpha_1^{-1}\alpha_2=\rho$ is a complex
number with rational components, and $\rho^2=(s_1s_2^{-1})^2
=z_1z_2^{-1}$. Therefore, $z_1z_2=(\rho z_2)^2$. The ring of
Gaussian integers is integrally closed in the field of
Gaussian numbers, so $\rho z_2\in\BG$. Since $z_1$ and $z_2$
are square-free, each Gaussian prime divisor of $z_1$ has
multiplicity $1$ in both $z_1$ and $z_2$, with the same
holding for $z_2$.  We see that $z_1$ and $z_2$ are
associates, and $\rho^2$ is a unit in $\BG$. Thus, $\rho$ is
a unit in $\BG$, establishing $(3)$ of
Theorem~\ref{repr_twins}.

Finally, we prove~$(4)$. Suppose that $(\alpha,z)$
parameterizes the twin pair $(\theta,\eta)$. Then
$\Norm(\theta)=\Norm(\alpha)^2\Norm(z)$ is a square if and
only if $\Norm(z)$ is a square. Clearly, if $z\in\BZ$ or
$iz\in\BZ$, then $\Norm(z)$ is a square.

Suppose that $\Norm(z)$ is a square and consider a Gaussian
prime divisor $\pi$ of $z$. As $z$ is square-free in~$\BG$,
the number $\pi^2$ does not divide $z$. We show that
$\pi=1+i$ is impossible. Indeed, all other Gaussian primes
have odd norm, so $\Norm(z)$ would have to be of the form
$4k+2$, which cannot be the square of an integer. Similarly,
if $p=\Norm(\pi)$ is an odd prime (of the from $4k+1$), then
$p\mid \Norm(z)$, and so the conjugate of $\pi$ must also be
a factor of $z$. Therefore, $z$ is indeed either real or
pure imaginary, and the proof of Theorem~\ref{repr_twins} is
complete.\qed

\begin{thm}\label{count_twins}
  For a positive integer $M$, denote by $\Twin(M)$ the number
  of twin pairs $(\theta,\eta)$ such that
  $\Norm(\theta)=\Norm(\eta)=M$, and let $\sigma(s)$ be the
  sum of positive integer divisors of any integer~$s$.
  Suppose that
\[
M=2^\kappa p_1^{\lambda_1}\ldots p_m^{\lambda_m}
q_1^{\mu_1}\ldots q_\ell^{\mu_\ell}\,,
\]
where $p_1,\ldots,p_m$ are primes $\equiv 1~(4)$ and
$q_1,\ldots,q_\ell$ are primes $\equiv -1~(4)$. We assume
that all $\lambda_r$ and $\mu_s$ are positive. Then
\[
\Twin(M)=24 \prod_{r=1}^m g(p_r^{\lambda_r})
\prod_{s=1}^\ell h(q_s^{\mu_s})\,,
\]
where
\[
g(p^{2\lambda})=
\sigma(p^{\lambda})+\sigma(p^{\lambda-1})\,,
\qquad
g(p^{2\lambda+1})=2\sigma(p^{\lambda})\,,
\]
and
\[
h(q^{2\mu})=\sigma(q^{\mu})+\sigma(q^{\mu-1})\,,
\qquad
h(q^{2\mu+1})=0\,.
\]
In particular, $\Twin(M)/24$ is a multiplicative function.
If there exists a twin pair with norm~$M$, then $M$ is the
sum of two squares.
\end{thm}

\begin{proof}
  By Theorem~\ref{repr_twins}, we have to count the number of
  pairs $(\alpha,z)\in\BE\times\BG$, where
  $M=\Norm(\alpha)^2\Norm(z)$ and $z$ is square-free
  in~$\BG$, and divide the number of solutions by~$4$
  due to $(3)$.

  Writing $z$ as a product of Gaussian primes, we see that in
  the canonical form of $\Norm(z)$ every prime of the form
  $4k+3$ has exponent~$2$, every prime of the form $4k+1$
  has exponent $2$ or~$1$, and the prime~$2$ has
  exponent~$1$ (or~$0$). If such a number $t=\Norm(z)$ is
  given, then the only freedom in determining~$z$ occurs at
  the primes $p$ of the form~$4k+1$. Indeed,
  $p=\pi_1\konj{\pi_1}$ is a product of two Gaussian primes,
  and if the exponent of~$p$ in~$\Norm(z)$ is~$1$, then we
  can decide whether to put $\pi$ or~$\konj{\pi}$ into~$z$.
  We have to multiply the resulting~$z$ with the four
  Gaussian units. Thus, if $4f(t)$ denotes the number of
  solutions for $z$ with norm~$t$, then $f$ is a
  multiplicative function, which is $1$ or~$0$ for every
  prime power, except that $f(p^1)=2$ when $p\equiv 1~(4)$.

  Corollary~\ref{quat_int_count} allows us to count the
  number of integral quaternions~$\alpha$ with given
  norm~$\Norm(\alpha)$. The result is $24$ times a
  multiplicative function (the sum of odd divisors of
  $\Norm(\alpha)$). We now go through all primes in the
  decomposition of~$M$ to see how we can split $M$ into
  $\Norm(\alpha)^2\Norm(z)$.

  If the prime is $2$, then we must put $2^\kappa$ into
  $\Norm(\alpha)^2$ when $\kappa$ is even, and must put
  $2^{\kappa-1}$ into~$\Norm(\alpha)^2$ when $\kappa$ is odd.
  By Corollary~\ref{quat_int_count}, we see that
  $\Twin(2^\kappa)=24$, and $\Twin(M)$ does not depend on
  $\kappa$.

  Next, we consider a prime $q_r\equiv -1~(4)$. In this case, 
  $\mu_r$ must be even for a solution to exist, and we can 
  either put the entire $q_r^{\mu_r}$~into $\Norm(\alpha)^2$ 
  or put $q_r^{\mu_r-2}$ into~$\Norm(\alpha)^2$ and $q_r$
  into~$z$. This proves the formula in the theorem for~$h$.

  Finally, for $p_r\equiv 1~(4)$ there are two cases to
  consider. If $\lambda_r$ is even, then we can put $0$ or
  $2$ copies of $p_r$~into $\Norm(z)$. If $\lambda_r$ is
  odd, then we must put $1$ copy of $p_r$ into~$\Norm(z)$
  (and the corresponding Gaussian primes in $z$ can be
  chosen in $2$ ways). This proves the formula for~$g$.

  Since we can put together the solutions for $M$ from the
  solutions for the prime divisors of $M$ independently, we
  get the formula in the theorem.
\end{proof}

\begin{cor}\label{twin_to_icube}
  Let $(v,w)$ be a pair of twins in $\BZ^3$ whose length is
  an integer. Then there exists an $\alpha\in\BE$ and an
  integer~$d$ such that the last two columns of $dE(\alpha)$
  are either $(v,w)$ or $(-w,v)$. Therefore, $(v,w)$ can be
  extended to an icube.
\end{cor}

\begin{proof}
  Let $\big(V(v),V(w)\big)=(\theta,\eta)$ be parameterized by
  $(\alpha,z)$, where $z$ is square-free. By $(4)$ of
  Theorem~\ref{repr_twins}, $z$ is either real or pure
  imaginary, so $z=d$ or $z=di$ for some integer~$d$. In the
  first case, $\theta=d\alpha j\,\konj{\alpha}$ and
  $\eta=d\alpha k\,\konj{\alpha}$, so the last two columns
  of $dE(\alpha)$ are $v$ and $w$. In the second case,
  $\theta=d\alpha k\,\konj{\alpha}$ and $\eta=-d\alpha
  j\,\konj{\alpha}$, so the last two columns of $dE(\alpha)$
  are $-w$ and $v$.
\end{proof}

\begin{cor}\label{primitive_icube_repr}
  If $(u,v,w)$ is an icube, then there is an $\alpha\in\BE$
  and $d\in\BZ$ such that $(u,v,w)$ and $dE(\alpha)$ can be
  obtained from each other by permuting and changing the signs
  of certain columns.
\end{cor}

\begin{proof}
  By Proposition~\ref{icube_length_int}, the edge length is
  an integer. Let $\alpha$ and $d$ be given by
  Corollary~\ref{twin_to_icube}. Then the columns of
  $dE(\alpha)$ and $(\pm u,\pm v,\pm w)$ share two
  orthogonal vectors, and so they share the third column of
  $E(\alpha)$ as well.
\end{proof}

\begin{lm}\label{twin_N(z)_sqfree}
  Suppose that $\alpha\in\BE$, $z$ is a square-free
  Gaussian integer and $\theta=\alpha zj\,\konj{\alpha}$ is
  primitive.  Then $\Norm(z)$ is the square-free part of\/
  $\Norm(\theta)$.
\end{lm}

\begin{proof}
  As $\Norm(\theta)=\Norm(\alpha)^2\Norm(z)$, it is
  sufficient to prove that $\Norm(z)$ is square-free.
  Suppose that $p^2\mid\Norm(z)$, for a prime $0<p\in\BZ$. If
  $p\equiv 3~(4)$, then $p$ is a Gaussian prime, so $p\mid z$,
  contradicting the fact that $\theta$ is primitive.  If
  $p=2$, then $2\mid z$, a contradiction.  Finally, if
  $p=\pi\konj{\pi}$, for some Gaussian prime $\pi$, then $\pi$
  and $\konj{\pi}$ cannot both divide $z$, because then $p$
  would divide the primitive $\theta$. On the other hand,
  the exponent of $\pi$ and $\konj{\pi}$ is at most~$1$
  in~$z$, since $z$ is square-free. Therefore, $\Norm(z)$
  cannot be divisible by~$p^2$, a contradiction.
\end{proof}

\begin{proof}[Proof of
  Theorem~\ref{primitive_unique_big_thm}]
  We have proved the existence statement in
  Remark~\ref{3repr_rem}. Suppose that $x$ is primitive.  If
  $(u,v,w)$ is an icube with edge length~$m$ such that
  $x=au+bv+cw$, then $(u,v,w)$ is also primitive. Therefore,
  by Corollary~\ref{primitive_icube_repr} (or by
  Corollary~\ref{sarkozy_cor}), we may assume that
  $(u,v,w)=E(\alpha)$, for some $\alpha\in\BE$.  Thus, it is
  sufficient to deal with ``Eulerian'' cubic lattices.

  Suppose that $x$ is contained in two such sublattices:
  $V(x)=\alpha_1\beta_1\,\konj{\alpha_1}=
  \alpha_2\beta_2\,\konj{\alpha_2}$. By the uniqueness part
  of Theorem~\ref{3repr}, there exists a unit $\vep\in\BE$
  such that $\alpha_2=\alpha_1\vep^{-1}$, and $\beta_2$,
  $\beta_1$ are group-conjugates via~$\vep$.
  Propositions~\ref{matrix_permute} and~\ref{unit_group}
  show that the matrices $E(\alpha_1)$ and $E(\alpha_2)$ may
  differ only by permutations and sign changes of columns.
  Therefore, the two cubic lattices are actually the same,
  proving the uniqueness part of the theorem. 
\end{proof}

\begin{proof}[Proof of
  Corollary~\ref{primitive_unique_big_cor}]
  Using the notation of the previous proof, let $L$ denote
  the unique sublattice obtained there.
  Proposition~\ref{unit_group} shows that $\beta_1$ and
  $\beta_2$ have the same number of zero components.
  Therefore, when considering the three cases of
  Corollary~\ref{primitive_unique_big_cor} (which are
  distinguished by the number of zero components of~$x$
  relative to~$L$), it does not matter which generating
  $\alpha$ we choose for~$L$.

  We now show that every twin of~$x$ is contained in $L$. Let
  $\eta_1$ be a twin of $\theta=V(x)$, and let
  $(\alpha_1,z_1)$ parameterize the pair $(\theta,\eta_1)$,
  with $z_1$ square-free. By
  Lemma~\ref{twin_N(z)_sqfree}, $\Norm(z_1)=n$, and so
  $\Norm(\alpha_1)=m$. Thus, the sublattice generated
  by~$\alpha_1$ is~$L$. Since $z_1j$ has a zero component,
  by case $(1)$ of Corollary~\ref{primitive_unique_big_cor},
  the vector~$x$ cannot have a twin.

  If the norm of~$x$ is a square, then $1=n=\Norm(z_1)$, and
  $z_1\in\{\pm 1, \pm i\}$. Thus, $x$~and each of
  its twins has exactly one nonzero component relative
  to~$L$. Therefore, $x$ has exactly~$4$ twins. Conversely,
  if $x$ has only one nonzero component relative to~$L$,
  then its length is obviously an integer, since the same
  holds for the generating vectors of~$L$. Hence, $(3)$ is
  proved.

  Finally, suppose that none of the components of~$z_1$ is
  zero (so $x$ has exactly one zero component). Let $\eta_2$
  be another twin of $\theta$, parameterized by
  $(\alpha_2,z_2)$. Again, $\alpha_2=\alpha_1\vep^{-1}$ and
  $\vep z_1j\vep^{-1}=z_2j$, for some unit~$\vep$; however,
  not every unit $\vep$ yields a twin of~$\theta$. Indeed,
  if $\vep\notin Q=\{\pm 1,\pm i, \pm j, \pm k\}$, then
  Proposition~\ref{unit_group} shows that group-conjugation
  by $\vep$ induces a fixed point free permutation on the
  components of the vectors (while possibly changing some
  signs). We know that the first component of $z_1 j$ and of
  $z_2 j$ is zero.  Therefore, $\vep\notin Q$ can happen only
  if $z_1 j$ and $z_2 j$ have two nonzero components, which
  we have excluded. Thus, $\vep\in Q$. The two twins of
  $\theta$ in question are $\eta_1=\alpha_1
  z_1k\,\konj{\alpha_1}$ and
\[
\eta_2=\alpha_2 z_2k\,\konj{\alpha_2}=
-\alpha_1 z_1(j\vep^{-1}i\vep)\konj{\alpha_1}\,.
\]
Let us calculate this now.

If $\vep\in\{\pm 1,\pm i\}$, then $\eta_2=\eta_1$. (This
has been noted in $(3)$ of Theorem~\ref{repr_twins}.)

If $\vep\in \{\pm j, \pm k\}$, then $\eta_2=-\eta_1$
(This is always obviously another twin of $\theta$.)

Thus, if $x$ has two nonzero components relative to~$L$, then
it has no more than two twins, completing the proof of
Corollary~\ref{primitive_unique_big_cor}.
\end{proof}

\section{Twin-complete numbers}
\label{sec_twin_compl}

\noindent
In this section we first prove the characterization of
twin-complete numbers given in
Theorem~\ref{twin_compl_full} and then discuss
Conjecture~\ref{Gauss-conj}.

\begin{lm}\label{twin_4n}
  If\/ $4n$ is twin-complete, then so is~$n$.
\end{lm}

\begin{proof}
  Every pure quaternion whose norm is divisible by $4$ is
  divisible by~$2$. This follows by looking at the
  coefficients mod~$4$ (or from
  Theorem~\ref{twin_quad_pow}).  Thus, if $\Norm(\theta)=n$,
  then $2\theta$ has a twin $\eta$, and so $\eta/2$ is a
  twin of~$\theta$.
\end{proof}

\begin{lm}\label{twin_yes_primitives}
  Let $\beta\in\BE$ be a primitive pure quaternion and $m>0$
  an odd positive integer. Then there exists an
  $\alpha\in\BE$ with norm~$m$ such that
  $\alpha\beta\,\konj{\alpha}$ is primitive.
\end{lm}

\begin{proof}
  It is sufficient to prove this when $m$ is a prime, since
  we can go through the prime divisors of~$m$ one by one.
  Clearly (or by Theorem~\ref{twin_quad_pow}),
  $\alpha\beta\,\konj{\alpha}$ is primitive if and only if
  it is not divisible by $p=m$. By
  Proposition~\ref{twin_no_primitives}, there is such an
  $\alpha$ (we have actually counted them).
\end{proof}

\begin{proof}[Proof of Theorem~\ref{twin_compl_full}]
  Let $n,m>0$ be integers, with $n$ square-free.
  Suppose that $n$ is twin-complete. It is sufficient
  to prove that every primitive vector $\delta$ with norm
  $nm^2$ has a twin (since we can do induction on~$m$). By
  Theorem~\ref{3repr}, $\delta=\alpha\beta\,\konj{\alpha}$,
  for some $\alpha,\beta\in\BE$ such that $\Norm(\alpha)=m$
  and $\Norm(\beta)=n$. Since $n$ is twin-complete, $\beta$
  has a twin~$\gamma$. We show that
  $\alpha\gamma\,\konj{\alpha}$ is a twin of $\delta$, using
  Proposition~\ref{twin_quat}. Indeed,
\[
\delta\alpha\gamma\,\konj{\alpha}=
\alpha\beta(\konj{\alpha}\alpha)\gamma\,\konj{\alpha}=
m\alpha\beta\gamma\konj{\alpha}\,.
\]
This is a pure quaternion, since $\beta\gamma$ is a pure
quaternion. This proves one direction of the theorem.

For the converse, suppose that $n>0$ is square-free and
$nm^2$ is twin-complete. Then every vector $\beta$ with norm
$n$ is primitive. By Lemma~\ref{twin_4n}, we may assume that
$m$ is odd. For any given~$\beta$,
Lemma~\ref{twin_yes_primitives} yields an $\alpha$ with
norm~$m$ such that $\theta=\alpha\beta\,\konj{\alpha}$ is
primitive. Since $nm^2$ is twin-complete, $\theta$ has a
twin. By Theorem~\ref{repr_twins}, this pair of twins can be
parameterized by some $(\alpha_1,z)\in\BE\times\BG$ such
that $z$ is square-free.  Thus, $\theta=\alpha_1
zj\,\konj{\alpha_1}$, and by Lemma~\ref{twin_N(z)_sqfree},
$\Norm(z)$ is the square-free part of $\Norm(\theta)$, that
is, $\Norm(z)=n$ and $\Norm(\alpha_1)=m$. By the uniqueness
statement of Theorem~\ref{3repr}, we get that $\beta=\vep
zj\vep^{-1}$, for some unit $\vep\in\BE$, but then $\vep
zk\vep^{-1}$ is a twin of $\beta$. Hence, $n$ is
twin-complete.

To prove the last statement of the theorem, let $n$ be square-free. 
Clearly, $(a,b,0)$ and $(-b,a,0)$ are twins, so if $n$
is not the sum of three positive squares, but is the sum of
two squares, then it is twin-complete.
Conversely, suppose that $n$ is twin-complete. Let $\beta$
be a pure quaternion of norm~$n$, we have to show that at
least one of the three coordinates of the corresponding
vector is zero.  As $n$ is twin-complete, $\beta$ has a
twin, so Theorem~\ref{repr_twins} implies that $\beta=\alpha
zj\,\konj{\alpha}$, for some $(\alpha,z)\in\BE\times\BG$.
Here $n=\Norm(\beta)=\Norm(\alpha)^2\Norm(z)$, so
$\Norm(\alpha)=1$ and $\alpha$ is a unit. From
Proposition~\ref{unit_group}, we get that at least one
component of $\beta$ is zero (since this is the case
with~$zj$).
\end{proof}

Now we discuss Conjecture~\ref{Gauss-conj}. Let
\[
S\supseteq \{1,  2, 5, 10, 13, 37, 58, 85, 130\}
\]
denote the list of those square-free numbers that can be
written as a sum of two squares, but not as
a sum of three positive squares. It has been known since
\cite{GCC59} that this list is finite, and if the conjecture
fails, there is at most one number in $S$ not listed above
(\cite{Wei73}, \cite{Gro85}).

In \cite{Mor60}, it is shown that for an integer $n\in S$,
the only nonnegative solutions of
\[
xy + yz + zx = n
\]
when $n \equiv 2~(4)$ are given by $xyz = 0$, and when $n
\equiv 1~(4)$ are given by either $xyz = 0$ or $x = d$, $y =
d$, $z = (n-d^2)/2d$, where $d$ is any divisor of $n$ with
$d^2 < n$. Either way, for such numbers $n$ the above
equation has no solution with three distinct positive
integers $x,\,y,\,z$.

This characterization of the numbers in~$S$ allows us to see
the relationship they bear with Euler's \emph{numeri
  idonei}.  Euler defined a \emph{numerus idoneus} to be an
integer~$N$ such that, for any positive integer $m$, if
\[
m = x^2 \pm Ny^2, \quad (x^2, Ny^2) = 1, \quad x,y \ge 0
\]
has a unique solution, then $m$ is of the form $2^ap^k$, $a
\in \{0,1\}$, $k \ge 1$, $p$ is a prime.

Euler was aware of 65 \textit{numeri idonei}, and it is
widely believed and conjectured that this list is complete
(\cite{Rib00}). S.~Chowla proved in \cite{Cho34} that there
are only finitely many \textit{numeri idonei}, and
P.~J.~Weinberger improved this result by showing that there
can be at most one more square-free idoneal number, and, if
it exists, it must be greater than $2\cdot10^{11}$
(\cite{Wei73}). If there is indeed another square-free
idoneal number $N$, and it is even, then $4N$ is also
idoneal. On the other hand, if $N > 1848$ is idoneal and not
square-free, then $N/4$ is both square-free and idoneal
(\cite{Kan09}). Thus, there are at most 67 idoneal numbers.

Using Theorem 3.22 of \cite{Cox89} it can be shown that an
integer $N$ is a \emph{numerus idoneus} if and only if it
cannot be expressed as $xy + yz + zx$ with $0 < x < y < z$.
Combining this with the characterization by \cite{Mor60}
described above, we see that every integer in~$S$ is also
one of Euler's \emph{numeri idonei}. Checking Euler's list
of the $65$ \emph{numeri idonei} (the greatest of which is
only $1848$) against the properties listed in
Conjecture~\ref{Gauss-conj}, one sees that, indeed,
Conjecture~\ref{Gauss-conj} is true if Euler's list is
complete.

If we only consider those integers $n$ for which there is no
representation of the form $xy + yz + zx$ with $1 \le x \le
y \le z$, i.e., the even, square-free, twin-complete numbers,
then we have from \cite{BC00} that such an $n$ can only be
absent from the list of Conjecture~\ref{Gauss-conj} if the
Generalized Riemann Hypothesis fails.

\end{document}